\documentclass[11pt,twoside]{article}

\usepackage{mathtools}%for some reason \coloneqq isnt working without this
\usepackage{a4wide}
\usepackage[latin1]{inputenc}
\usepackage{amsmath}
\usepackage{amsfonts}
\usepackage{amssymb}
\usepackage{amsthm}
\usepackage[svgnames]{xcolor}
\usepackage{tikz-cd}
\usepackage[left=2.00cm, right=2.00cm, top=2.00cm]{geometry}
\usepackage{tikz}
\usetikzlibrary{arrows}
\usetikzlibrary{arrows.meta}  
\usepackage{subcaption}  %allows subfigures
\usepackage[section]{placeins}%This prevents placing floats before a section.

\usepackage{fancyhdr}
\def\headcolour{\color{DarkGrey}}
\theoremstyle{definition}
\newtheorem{defn}{Definition}[section]
\theoremstyle{plain}
\newtheorem{thm}{Theorem}[section]
\newtheorem{cor}[thm]{Corollary}
\newtheorem{propn}[thm]{Proposition}
\numberwithin{equation}{section}

\DeclareMathOperator{\Imag}{Im}

\DeclareMathOperator{\Ad}{Ad}

\DeclareMathOperator{\Gr}{Gr}
\DeclareMathOperator{\spoon}{\text{Span}}

\DeclareMathOperator{\aff}{Aff\mathcal{F}}

\DeclareMathOperator{\F}{\mathcal{F}}
\newcommand{\orb}{\mathcal{O}}

\title{Hermitian flag manifolds and orbits of the Euclidean group}
\author{Philip Arathoon, James Montaldi}
\date{}

\begin{document}
	\maketitle

\begin{abstract}
	 	We study the adjoint and coadjoint representations of a class of Lie group including the Euclidean group. Despite the fact that these representations are not in general isomorphic, we show that there is a geometrically defined bijection between the sets of adjoint and coadjoint orbits of such groups. In addition, we show that the corresponding orbits, although different, are homotopy equivalent. We also provide a geometric description of the adjoint and coadjoint orbits of the Euclidean and orthogonal groups as a special class of flag manifold which we call a Hermitian flag manifold. These manifolds consist of flags endowed with complex structures equipped to the quotient spaces that define the flag.

\medskip


\end{abstract}

	\section{Introduction}
	
	For any Lie group, the adjoint and coadjoint representations are of fundamental importance to an understanding of the structure of the Lie group itself, as well as in various applications such as geometric mechanics, symplectic geometry, representation theory, and more. 
	
For the example of the unitary group, the adjoint orbits are well known to equal the isospectral sets of skew-Hermitian matrices, which in turn, may canonically be identified with manifolds of complex flags.  The coadjoint orbits of the unitary group are identical to the adjoint orbits as a consequence of there being an isomorphism between the adjoint and coadjoint representations. Indeed, the same is true for any compact or semisimple Lie group. However, a similar flag-like interpretation of (co-)adjoint orbits for $SO(n)$ appears to be absent from the literature---we provide such an interpretation (as \emph{Hermitian flags}) en route to studying the adjoint and coadjoint orbits of the Euclidean group . 

We therefore consider examples of groups which are neither compact nor semisimple; namely, we consider a class of semidirect products which we call \emph{groups of Euclidean type}. Specifically, such a group \(G\) is obtained by fixing a compact Lie group \(H\) together with a representation  \(V\) and forming the semidirect product \(G=H\ltimes V\) whose group product is given by 
\[
(h_1,v_1)(h_2,v_2)=(h_1h_2, v_1+h_1v_2).
\]
The Euclidean group \(E(n)\) is an example of such a group, given by setting \(H=O(n)\) and \(V=\mathbb{R}^n\) with the standard representation. For the particular example of the Euclidean and orthogonal groups, we classify the adjoint and coadjoint orbits, and exhibit these orbits as a \emph{manifolds of Hermitian flags}. These manifolds consist of regular flags which are additionally equipped with some extra structure. The quotient spaces which define a flag may be equipped with orientations or a metric preserving complex structure; hence our choice of the term \emph{Hermitian}. The well-established literature concerning the coadjoint orbits of a semidirect product, allow us to derive interesting results concerning the symplectic geometry of such flag manifolds.
	
The adjoint and coadjoint representations for the Euclidean group are not in general isomorphic. Therefore, there is no reason to suspect that both orbits should have much in common. However, for the example of the Poincar\'{e} group, obtained by setting \(H\) equal to an indefinite orthogonal group, there exists a ``curious bijection'' between the sets of adjoint and coadjoint orbits \cite{cush06}. This bijection may be visually demonstrated for the example of the special Euclidean group in two dimensions. Figure~\ref{se2} shows the adjoint and coadjoint orbits of \(SE(2)\). Observe how there exists a geometric bijection between orbits: cylinder coadjoint orbits to circle adjoint orbits, plane adjoint orbits to point coadjoint orbits, and both origins to each other. Moreover, two orbits corresponding under this bijection are homotopic. We will exhibit this bijection for all groups of Euclidean type, and prove that all orbits which correspond under the bijection are homotopic to each other. For the Euclidean group a stronger result is true: we show that for any two orbits in bijection, one is a vector or affine bundle over the other.

The paper is in two parts: the first (Section \ref{sec:orbit geometry}) describes the relation mentioned above between adjoint and coadjoint orbits of groups of Euclidean type, while in Section\,\ref{sec:Euclidean} we discuss the particular case of the Euclidean group itself, where more detailed geometry is  determined. 

Our results only apply to semidirect products of Euclidean type, and thus exclude the example of the Poincar\'{e} group on account of the indefinite orthogonal groups being non-compact. However, our methods can in fact be adapted so that they apply to a wider class of semidirect products including the Poincar\'{e} group. This is beyond the aims of this paper but will be included in a sequel \cite{Arathoon}.

	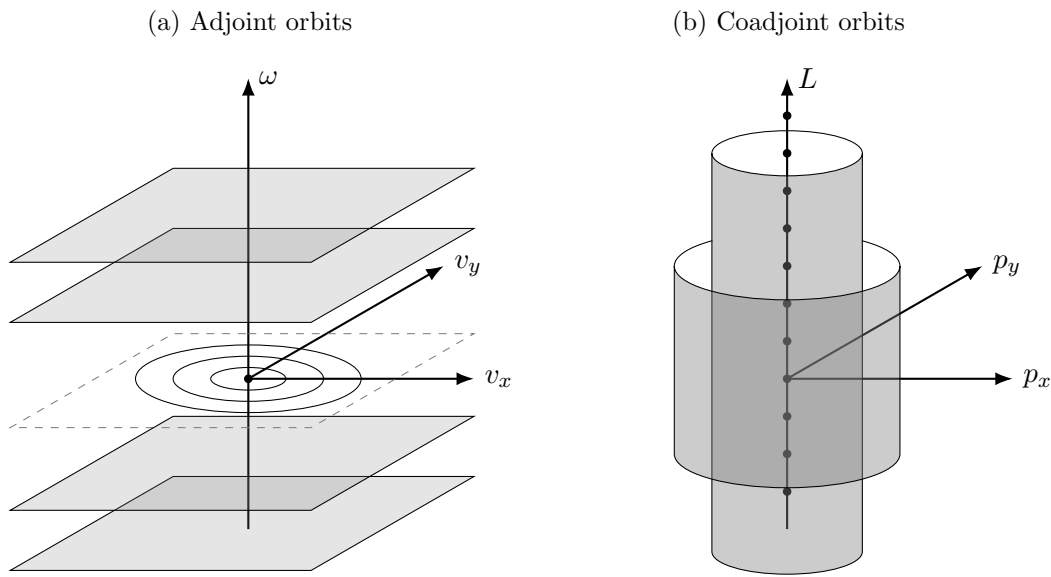
\begin{figure}[h]
		\centering
		\begin{subfigure}{0.4\textwidth}
			\caption{Adjoint orbits}
			\begin{tikzpicture}
			\path[use as bounding box, ] (-3.5,-3) rectangle (3.5,4.3);
			%axes
			\draw[thick,-Latex] (0,0)-- (3,0)node[right]{$v_x$};
			\draw[thick,-Latex] (0,0) -- (30:3)node[right]{$v_y$};
			\draw[thick,-Latex] (0,-2) -- (0,4)node[right]{$\omega$};
			
			%fixed points
			\draw[fill] (0,0) circle (0.05);
			
			%ground floor
			\draw[dashed,gray] (0,0.6) -- ++(3,0) -- ++(210:2.5) -- ++(-4,0) -- ++(30:2.5) --cycle;
			\draw circle (0.5 and 0.15);
			\draw circle (1 and 0.3);
			\draw circle (1.5 and 0.45);
			
			%upper floor
			\draw[fill,fill opacity=0.2,gray,draw=black] (0,2) -- ++(3,0) -- ++(210:2.5) -- ++(-4,0) -- ++(30:2.5) --cycle;
			
			%another upper floor?
			\draw[fill,fill opacity=0.2,gray,draw=black] (0,2.8) -- ++(3,0) -- ++(210:2.5) -- ++(-4,0) -- ++(30:2.5) --cycle;
			%another lower?
			\draw[fill,fill opacity=0.2,gray,draw=black] (0,-1.3) -- ++(3,0) -- ++(210:2.5) -- ++(-4,0) -- ++(30:2.5) --cycle;
			%lower floor
			\draw[fill,fill opacity=0.2,gray,draw=black] (0,-0.5) -- ++(3,0) -- ++(210:2.5) -- ++(-4,0) -- ++(30:2.5) --cycle;
			\end{tikzpicture}
		\end{subfigure}
		\begin{subfigure}{0.4\textwidth}
			\caption{Coadjoint orbits}
			\begin{tikzpicture}
			\path[use as bounding box, ] (-3.5,-3) rectangle (3.5,4.3);
			%axes
			\draw[thick,-Latex] (0,0)-- (3,0)node[right]{$p_x$};
			\draw[thick,-Latex] (0,0) -- (30:3)node[right]{$p_y$};
			\draw[thick,-Latex] (0,-2) -- (0,4)node[right]{$L$};
			
			%fixed points
			\foreach \x in {-1.5,-1,...,3.5} 
			\draw[fill] (0,\x) circle (0.05);
			
			%ends of cylinders
			\draw (1,3) arc (0:180:1 and 0.3);
			\draw[fill=gray,fill opacity=0.4,draw=black] (-1,3) arc (180:360:1 and 0.3) -- (1,-2.3) arc(0:-180: 1 and 0.3) -- cycle;
			\draw (1.5,1.5) arc (0:48.5:1.5 and 0.3*1.5);
			\draw (-1.5,1.5) arc (0:48.5:-1.5 and 0.3*1.5);
			\draw[fill=gray,fill opacity=0.4,draw=black] (-1.5,1.5) arc (180:360:1.5 and 0.3*1.5) -- (1.5,-1)arc(0:-180: 1.5 and 0.3*1.5) -- cycle;
			\end{tikzpicture}
		\end{subfigure}
\begin{minipage}{0.9\textwidth}
\caption{Orbits for the special Euclidean group \(SE(2)\) on its Lie algebra and dual. Let \((\omega,{v})\) denote elements in the Lie algebra \(\mathfrak{se}(2)=\mathfrak{so}(2)\times\mathbb{R}^2\), and \((L,p)\) for elements in the dual. The adjoint representation is given by \(\Ad_{(r,{d})}(\omega,{v})=(\omega,r={v}-\omega R_{\pi/2}{d})\) and the coadjoint representation by \(\Ad_{(r,{d})}^*(L,{p})=( L+{d}^T R_{\pi/2}{p},r{p})\). Here, \(R_{\pi/2}\) is an anticlockwise rotation by \(\pi/2\) and \((r,{d})\) is an element in \(SE(2)\).}
\label{se2} 
\end{minipage}
	\end{figure}
	
It is natural to ask how much of the bijection properties discussed in this paper extend to more general groups, and to other contragredient pairs of representations. 
	
	\subsection{Notation}
	Let \(G\) be a group and \(M\) a space upon which \(G\) acts, \(G\times M\rightarrow M;~(g,m)\mapsto gm\). We denote the orbit through \(m\in M\) by \(\orb^G_m\). The stabiliser or isotropy subgroup of this action at \(m\) is the subgroup \(\{g\in G~|~gm=m\}\) and is denoted by \(G_m\). The orbits are themselves homogeneous \(G\)-spaces which may be written as a coset space \(\orb^G_m\cong G/G_m\).
	
	Throughout this paper \(H\) will denote a Lie group with Lie algebra \(\mathfrak{h}\) and identity element \(e\), and \(V\) a representation of \(H\). From this we consider the semidirect product \(G=H\ltimes V\) as defined earlier with Lie algebra \(\mathfrak{g}=\mathfrak{h}\times V\). Our guiding example will be that of the Euclidean group \(G=E(n)\).
	
	\begin{defn}
		The semidirect product $G=H\ltimes V$ will be called a group of \emph{Euclidean type} if $H$ is compact. Since any compact group is a subgroup of an orthogonal group, we remark that every group of Euclidean type is a subgroup of the Euclidean group $E(n)$ for large enough $n$.
	\end{defn}

	 The elements of the Euclidean group, along with its Lie algebra and dual, have well known physical meanings together with standard notation. We mimic this notation to represent typical elements of our group \(G\).
	
	\begin{itemize}
		\item[\textendash] Group elements, \((r,d)\in G\); \(r\in H\) for \textit{rotation} and \(d\in V\) for \textit{displacement}.
		
		\item[\textendash] Lie algebra elements, \((\omega,v)\in \mathfrak{g}\); \(\omega\in \mathfrak{h}\) for \textit{angular velocity} and \(v\in V\) for \textit{linear velocity}.
		
		\item[\textendash] Dual Lie algebra elements, \((L,p)\in \mathfrak{g}^*\); \(L\in \mathfrak{h}^*\) for \textit{angular momentum} and \(p\in V^*\) for \textit{linear momentum}.
	\end{itemize}

\section{Orbit geometry} \label{sec:orbit geometry}
	In \cite{rawnsley} the isotropy subgroup for a coadjoint orbit is shown to be a group extension of a smaller group, referred to in the literature as a \emph{little group}. Here we derive an analogous result for the adjoint orbits, and restrict our attention to the case where the group extension is a split extension; whereby the isotropy subgroup is itself given by a semidirect product. With this condition on the orbits, a property we call \emph{properness}, we are able to exhibit a rich geometric structure between the orbits and their canonical bundles and submanifolds.
	\subsection{Isotropy subgroups }
	The dual \(\mathfrak{g}^*\) is canonically isomorphic to \(\mathfrak{h}^*\times V^*\). For any \(\eta\in\mathfrak{g}^*\) we may identify it with the pair \((L,p)\in\mathfrak{h}^*\times V^*\) which satisfies \(\langle \eta,(\omega,v)\rangle=\langle L,\omega\rangle+\langle p,v\rangle\) for every \((\omega,v)\) in \(\mathfrak{g}\), and where \(\langle~,~\rangle\) denotes the pairing between a space and its dual. The coadjoint action is then given by \cite{rawnsley}
	\begin{equation}
	\label{alt_Coad_action}
	\Ad_{(r,d)}^*(L,p)=\left(\Ad_r^*L+\mu\left(rp,d\right),rp\right).
	\end{equation}
	Here we have the momentum map \(\mu\colon V^*\times V\longrightarrow\mathfrak{h}^*\) satisfying
	\begin{equation}\label{mom_map}
	\langle\mu(p,v),\omega\rangle=\langle p,\omega v\rangle
	\end{equation}
	for all \(\omega\in\mathfrak{h}\). For a fixed \(p\in V^*\) introduce the map
	 	\begin{equation}\label{tau}
	 	\tau_p\colon V\longrightarrow\mathfrak{h}^*
	 	\end{equation}
	 	given by \(\tau_p(v)=\mu(p,v)\). The group \(H_p=\{r~|~rp=p\}\) is called the \emph{little group}. The Lie algebra of this group \(\mathfrak{h}_p=\{\omega~|~\omega p=0\}\) has the property that its annihilator \(\mathfrak{h}_p^\circ\) in \(\mathfrak{h}^*\) is equal to the image of \(\tau_p\) \cite[Lemma 1]{rawnsley}.
	 	
	 	For a given \((L,p)\) consider the isotropy subgroup \(G_{L,p}\). From \eqref{alt_Coad_action}, for \((r,d)\) to belong to \(G_{L,p}\) we must first have \(rp=p\), and therefore that \(r\) belongs to \(H_p\); and secondly that \(\Ad_r^*L+\tau_p(d)=L\). Project both sides of the second equation onto \(\mathfrak{h}_p^*\) using the canonical restriction map \(\iota^*_p\colon\mathfrak{h}^*\rightarrow\mathfrak{h}_p^*\). Since \(\Imag\tau_p=\mathfrak{h}_p^\circ=\ker\iota^*_p\), and given that \(\iota^*_p\) commutes with the coadjoint action restricted to \(H_p\), we have that \(\Ad^*_r(\iota^*_pL)=\iota^*_pL\). Therefore \(r\) belongs to the subgroup \({(H_p)}_{\iota^*_pL}\) and \(G_{L,p}\) fits into the exact sequence
\begin{equation}
\label{Coad_isotrop}
\{0\}\longrightarrow\ker\tau_p\overset{i}{\longrightarrow}G_{L,p}\overset{j}{\longrightarrow}\left(H_p\right)_{\iota^*_pL}\longrightarrow\{e\}.
\end{equation}
For this sequence we have \(i(d)=(e,d)\) and \(j(r,d)=r\). This sequence is not usually split. If however, the inclusion \(\sigma\colon \left(H_p\right)_{\iota^*_pL}\longrightarrow G\) given by \(\sigma(r)=(r,0)\) defines a homomorphism into \(G_{L,p}\), then we have a split exact sequence, and the isotropy subgroup is equal to
\begin{equation}
\label{Coad_proper_isotropygroup}
G_{L,p}=\left(H_p\right)_{\iota^*_pL}\ltimes\ker\tau_p.
\end{equation}
This condition is satisfied whenever \(\left(H_p\right)_{\iota^*_pL}\times\{0\}\) belongs to \(G_{L,p}\), or equivalently when
\begin{equation}\label{proper_coad}
\left(H_p\right)_{\iota^*_pL}=H_p\cap H_L.
\end{equation}
We now provide an analogous result for the adjoint orbits. Let \((\omega,v)\) belong to \(\mathfrak{g}\) and consider the expression for the adjoint action \cite[Section 19]{stern},
\begin{equation}
\label{alt_Ad_action}
\Ad_{(r,d)}(\omega,v)=\left(\Ad_r\omega,rv-(\Ad_r\omega)d\right).
\end{equation}
If \((r,d)\) is to belong to the stabiliser \(G_{\omega,v}\) then we firstly require \(r\in H_\omega\). We then must have \(rv-\omega d=v\). Project both sides of this equation onto the quotient space \(V/\Imag\omega\). This projection commutes with the representation of \(H_\omega\) on \(V/\Imag\omega\) (note that this is well defined since \(\Imag\omega\) is invariant under \(H_\omega\)). Hence \(r[v]=[v]\) and so we must additionally have that \(r\) belongs to the subgroup \((H_\omega)_{[v]}\). The stabiliser \(G_{\omega,v}\) then fits into the exact sequence
\begin{equation}\label{adjoint_isotrop}
\{0\}\longrightarrow\ker\omega\overset{i}{\longrightarrow}G_{\omega,v}\overset{j}{\longrightarrow}\left(H_\omega\right)_{[v]}\longrightarrow\{e\}.
\end{equation}
Here once again, \(i(d)=(e,d)\) and \(j(r,d)=r\). As with the coadjoint isotropy subgroups, this exact sequence is not necessarily split. If however, there exists a splitting map \(\sigma\colon (H_\omega)_{[v]}\longrightarrow G_{\omega,v}\) given by \(\sigma(r)=(r,0)\), then \(G_{\omega,v}\) is equal to the semidirect product
\begin{equation}
\label{Ad_isotropygroup_proper}
G_{\omega,v}=(H_\omega)_{[v]}\ltimes\ker\omega.
\end{equation}
Equivalently, this is the case when
\begin{equation}\label{proper_ad}
(H_\omega)_{[v]}=H_\omega\cap H_v.
\end{equation}
\begin{defn}
	If equations~{\eqref{proper_ad}} and~{\eqref{proper_coad}} hold for given points $(\omega,v)\in\mathfrak{g}$ and $(L,p)\in\mathfrak{g}^*$, then these points together with the orbits through them will be called \emph{proper}. This is equivalent to the isotropy subgroups being equal to the semidirect products in \eqref{Ad_isotropygroup_proper} and~\eqref{Coad_proper_isotropygroup}.
\end{defn}
\subsection{Bundles and submanifolds}
The orbit through any point contains two natural submanifolds given by restricting the action through this point to the two subgroups, \(H=H\times\{0\}\), and \(V=\{e\}\times V\). The purpose of this subsection is to exhibit an interplay between these submanifolds and the natural bundle structures defined between orbits. 

\begin{thm}
	\label{orbit_geometry_thm}
	For when the points \((\omega,v)\in\mathfrak{g}\) and \((L,p)\in\mathfrak{g}^*\) are proper, the \(G\)- and \(H\)-orbits are connected by the bundle maps in Figure~\ref{orbit_geometry_fig}. By insisting that \(V\) act trivially on the \(H\)-spaces in the diagram, all maps are \(G\)-equivariant with fibres equivariantly diffeomorphic to the space written next to the arrow. The horizontal arrows are vector bundles with fibres equal to the orbits of \(V\). For the coadjoint orbit \(\orb^G_{L,p}\) equipped with the Kirillov-Kostant-Souriau (K.K.S.) symplectic form, the fibres of \(\orb^G_{L,p}\rightarrow\orb^H_{L,p}\) are isotropic submanifolds, and the fibres of \(\orb^G_{L,p}\rightarrow\orb^G_{0,p}\) are symplectic submanifolds.
\end{thm}
\begin{proof}
	
	By properness of \((\omega,v)\) and \((L,p)\) we can use equations \eqref{Ad_isotropygroup_proper} and \eqref{Coad_proper_isotropygroup} to express the isotropy subgroups of all the spaces. The equivariant maps in the diagrams then follow. For instance: \(G_{\omega,v}=(H_\omega\cap H_v)\ltimes\ker\omega\subset(H_\omega\cap H_v)\ltimes V=G_{\omega,0}\) and so the map \(\Ad_{(r,d)}(\omega,v)\mapsto\Ad_{(r,d)}(\omega,0)\) gives a well-defined equivariant bijection \(\orb^G_{\omega,v}\rightarrow\orb^G_{\omega,0}\). That the horizontal fibres are vector bundles follows from the fact that the fibres are equal to the orbits of \(V\) which are isomorphic to \(\Imag\omega\) and \(\mathfrak{h}^\circ_p\), and that the image of the bundle map naturally embeds as a zero section.

	Now to turn our attention to the symplectic geometry of the coadjoint orbits. That the orbits of \(V\) in \(\orb^G_{L,p}\) are isotropic submanifolds is proven in \cite[Proposition 4.4]{baguis}. Finally, for \(r\in H_p\) the map \(\Ad_{(r,0)}^*(L,p)\mapsto\Ad_r^*(\iota^*_pL)\) defines an equivariant diffeomorphism between the fibre of \(\orb^G_{L,p}\rightarrow\orb^G_{0,p}\) over \((0,p)\), and the coadjoint orbit of \(H_p\) through \(\iota^*_pL\in\mathfrak{h}_p^*\). It is a straightforward calculation to check that the pullback of the K.K.S. form on the coadjoint orbit in \(\mathfrak{h}_p^*\) agrees with the K.K.S. form on \(\orb^G_{L,p}\) restricted to the fibre.
\end{proof}

\begin{figure}[h]
	\centering
	\begin{subfigure}{0.35\textwidth}
		\begin{tikzpicture}[scale=0.6]
		\path[use as bounding box] (-2,-1) rectangle (8,7);
		%%ALL the diagram arrows and spaces
		\draw[-{Latex}] (0.9,0) node [left,scale=1.4]{$\orb^G_{\omega,0}$} -- (5.4,0) node[right,scale=1.4]{$\orb^H_\omega$};
		\draw[-{Latex} ] (6,5.1) node[scale=1.4,above] {$\orb^H_{\omega,v}$} -- (6,.6);
		\draw[-{Latex}  ] (0,5.4)  -- (0,0.8);
		\draw[-{Latex}] (1.4,6)node[left, scale=1.4] {$\orb^G_{\omega,v}$} -- (5,6);
		%%% the annotation labels
		\node at (3,6.5){$\orb^V_{\omega,v}\cong\Imag\omega$};
		\node at (3,-0.5){$\orb^V_{\omega,0}\cong\Imag\omega$};
		\node at (7,3){$\orb^{H_\omega}_{[v]}$};		
		\node at (-1,3)	{$\orb^{H_\omega}_{[v]}$ };
		\end{tikzpicture}
	\end{subfigure}
	\begin{subfigure}{0.35\textwidth}
		\begin{tikzpicture}[scale=0.6]
		\path[use as bounding box](-2,-1) rectangle (8,7);
		%%ALL the diagram arrows and spaces
		\draw[-{Latex}] (0.9,0) node [left,scale=1.4]{$\orb^G_{0,p}$} -- (5.4,0) node[right,scale=1.4]{$\orb^H_p$};
		\draw[-{Latex} ] (6,5.1) node[scale=1.4,above] {$\orb^H_{L,p}$} -- (6,.6);
		\draw[-{Latex}  ] (0,5.4)  -- (0,0.8);		
		\draw[-{Latex}] (1.4,6)node[left, scale=1.4] {$\orb^G_{L,p}$} -- (5,6);		
		%%% the annotation labels		
		\node at (3,6.5){$\orb^V_{L,p}\cong\mathfrak{h}^\circ_p$};		
		\node at (3,-0.5){$\orb^V_{0,p}\cong\mathfrak{h}^\circ_p$};		
		\node at (7,3){$\orb^{H_p}_{\iota^*_pL}$};		
		\node at (-1,3)	{$\orb^{H_p}_{\iota^*_pL}$ };
		%\node[scale=1.4] at (0.05,-1){\rotatebox[]{90}{$\cong$}};
		%\node[scale=1.4] at (0.2,-2){$T^*\orb^H_p$};
		\end{tikzpicture}		
	\end{subfigure}		
\caption{\label{orbit_geometry_fig}}
\end{figure}

\subsection{An orbit bijection}

From here on we suppose that \(H\) is compact and so \(G\) is a group of Euclidean type. Using the standard averaging arguments, we may equip each of \(V\) and \(\mathfrak{h}\) with an \(H\)-invariant inner product; say \(B_V\) and \(B_{\mathfrak{h}}\) respectively. With this additional structure we may establish \(H\)-equivariant isomorphisms between \(V\) and \(\mathfrak{h}\) with their duals. These isomorphisms will be given by sending \(v\) and \(\omega\) to the elements \(p(v)\) and \(L(\omega)\) satisfying \(B_V(v,x)=\langle p(v),x\rangle\) and \(B_\mathfrak{h}(\omega,\xi)=\langle L(\omega),\xi\rangle\) for all \(x\in V\) and \(\xi\in\mathfrak{h}\).
\begin{defn}
	Define the subset \(\Delta\subset\mathfrak{g}\) by
	\begin{equation}
	\Delta\coloneqq\left\{(\omega,v)~|~\omega v=0\right\}.
	\end{equation}
	The image of \(\Delta\) under the \(H\)-equivariant isomorphism \(\varphi\colon\mathfrak{g}\rightarrow\mathfrak{g}^*\) given by \(\varphi(\omega,v)=\left(L(\omega),p(v)\right)\) is the set \(\varphi(\Delta)\coloneqq\Delta^*\). We will call these the \emph{Cartan subsets}.
\end{defn}

 In the next proposition we show that these sets serve as normal forms for the orbits. This result explains our choice of terminology; the Cartan subsets are analogous to the Cartan subspace of a compact Lie algebra, through which all adjoint orbits intersect \cite[Chapter 4]{bott}.

\begin{propn}\label{delta_sets}
	The adjoint and coadjoint orbits of a group \(G\) of Euclidean type intersect the sets \(\Delta\) and \(\Delta^*\) respectively. Furthermore, all points belonging to these sets are proper and therefore, all orbits of \(\mathfrak{g}\) and \(\mathfrak{g}^*\) are proper. Finally, for any \((\omega,v)\in\Delta\) and \((L,p)\in\Delta^*\), we have \(\orb^G_{\omega,v}\cap\Delta=\orb^{H}_{\omega,v}\) and \(\orb^G_{L,p}\cap\Delta^*=\orb^{H}_{L,p}\). 
\end{propn}
\begin{proof}
	As \(\omega\) is skew-self-adjoint with respect to \(B_V\), \(V\) admits an orthogonal decomposition \(V=\Imag\omega\oplus\ker\omega\). With reference to the adjoint action in \eqref{alt_Ad_action}, observe that any adjoint orbit must therefore intersect \(\Delta\), and that this intersection is equal to an orbit of \(H\). As the decomposition \(V=\Imag\omega\oplus\ker\omega\) is orthogonal with respect to the inner product \(B_V\), the inclusion defines an \(H_\omega\)-equivariant isomorphism between \(\ker\omega\) and the quotient \(V/\Imag\omega\). Thus, for any \((\omega,v)\) in \(\Delta\) we have \(v\in\ker\omega\), and so \((H_\omega)_{[v]}=H_\omega\cap H_v\). Therefore, every orbit in \(\mathfrak{g}\) is proper.
	
	For the coadjoint orbits we claim to have the \(H_p\)-invariant decomposition \(\mathfrak{h}^*=\mathfrak{h}_p^\circ\oplus L(\mathfrak{h}_p)\) for any \(p\in V^*\). This holds because \(L(\mathfrak{h}_p)\) is the orthogonal complement to \(\mathfrak{h}_p^\circ\) with respect to the inner product \(B_\mathfrak{h}\). With reference to equation~\eqref{alt_Coad_action} and the fact that \(\Imag\tau_p=\mathfrak{h}_p^\circ\), one sees that every orbit in \(\mathfrak{g}^*\) intersects \(\Delta^*\), and that this intersection is equal to an orbit of \(H\). The map \(\iota^*_p\) restricted to \(L(\mathfrak{h}_p)\) defines an \(H_p\)-equivariant isomorphism between \(L(\mathfrak{h}_p)\) and \(\mathfrak{h}_p^*\). Hence, for any \((L,p)\) in \(\Delta^*\) we have \((H_p)_{\iota^*_pL}=H_p\cap H_{L}\) and therefore all points in \(\Delta^*\) are proper.
\end{proof}
As the bijection \(\varphi\colon\Delta\rightarrow\Delta^*\) is \(H\)-equivariant, it follows that the set of \(H\)-orbits in \(\Delta\) is in bijection with the set of \(H\)-orbits in \(\Delta^*\). However, as a corollary to Proposition~\ref{delta_sets} the sets of adjoint and coadjoint orbits of \(G\) are in bijection with the sets of \(H\)-orbits in \(\Delta\) and \(\Delta^*\) respectively, and therefore by extension, so too are orbits in \(\mathfrak{g}\) with \(\mathfrak{g}^*\).

\begin{thm}\label{bijection}
	The map \(\varphi\colon\Delta\rightarrow\Delta^*\) establishes a bijection between the sets of adjoint and coadjoint orbits of \(G\). More precisely, the adjoint orbit through \((\omega,v)\in\Delta\) corresponds under the bijection with the coadjoint orbit through \(\varphi(\omega,v)=(L(\omega),p(v))\in\Delta^*\). Furthermore, any two orbits corresponding under the bijection are vector bundles over a common base space, namely the manifold \(\orb^H_{\omega,v}\cong\orb^H_{L,p}\). In particular, orbits in bijection with each other have the same homotopy type.
\end{thm}
\begin{proof}
	Take \((\omega,v)\in\Delta\) and write \((L(\omega),p(v))\in\Delta^*\) as \((L,p)\). From Theorem~\ref{orbit_geometry_thm} the orbits \(\orb^G_{\omega,v}\) and \(\orb^G_{L,p}\) are both vector bundles over \(\orb^H_{\omega,v}\) and \(\orb^H_{L,p}\) with fibres isomorphic to \(\Imag\omega\) and \(\mathfrak{h}_p^\circ\) respectively. Since \(\omega\mapsto L(\omega)\) and \(v\mapsto p(v)\) are \(H\)-equivariant maps, \(H_\omega=H_L\) and \(H_v=H_p\). It follows that \(H_\omega\cap H_v=H_p\cap H_L\) and therefore \(\orb^H_{\omega,v}\cong\orb^H_{L,p}\).
\end{proof}
We obtain a stronger bijection result for the particular example of the (special) Euclidean group. We will show that for any two orbits corresponding under the bijection, one is either a vector or affine bundle over the other. To do this we need to first refine the Cartan sets by decomposing them into disjoint \(H\)-invariant subsets \(\Delta=\Delta_\mathfrak{h}\cup\overline{\Delta}\) and \(\Delta^*=\Delta_{\mathfrak{h}^*}\cup\overline{\Delta^*}\) given by
\begin{equation}\label{Delta_defn}
\begin{array}{ll}
\Delta_{\mathfrak{h}}=\left\{(\omega,0)~|~\omega\in\mathfrak{h}\right\}, & \overline{\Delta}=\left\{(\omega,v)\in\Delta~|~v\ne 0\right\}, \\
\Delta_{\mathfrak{h}^*}=\left\{(L,0)~|~\omega\in\mathfrak{h}^*\right\}, &	\overline{\Delta^*}=\left\{(L,p)\in\Delta^*~|~p\ne 0\right\}.
\end{array}
\end{equation}
Orbits through \(\overline{\Delta}\) and \(\overline{\Delta^*}\) might be described as being \emph{generic} orbits. Also note that the bijection pairs orbits in \(\Delta_{\mathfrak{h}}\) and \(\Delta_{\mathfrak{h}^*}\) with each other, likewise with \(\overline{\Delta}\) and \(\overline{\Delta^*}\).

Secondly, we recall that for a principal \(G\)-bundle \(P\rightarrow B\) (where \(G\) may now be any group) together with a space \(F\) upon which \(G\) acts, the associated fibre bundle \(B_F=P\times_G F\) is the set 
\[
	\text{$P\times F/\sim$, where $(p,f)\sim (pg^{-1},gf)$.}
\]
We obtain the following proposition for the special case when the base space \(B\) is a homogeneous space and where the fibre \(F\) is acted upon transitively.
\begin{propn}\label{associated_vector}
	Let \(B\) be a homogeneous \(G\)-space and \(G_b\) the isotropy subgroup for a given \(b\in B\). The map \(g\mapsto gb\) defines a principal \(G_b\)-bundle \(G\rightarrow B\). Suppose \(F\) is a space upon which \(G_b\) acts transitively. Then the associated fibre bundle \(B_F=G\times_{G_b}F\) admits a transitive action of \(G\) with isotropy subgroup isomorphic to \((G_b)_f\) for some \(f\in F\), and for which the map \(B_F\rightarrow F\) is \(G\)-equivariant.
\end{propn}
\begin{proof}
	The action of \(G\) on \(B_F\) is given by \(\widetilde{g}\cdot\left[(g,f)\right]=\left[(\widetilde{g}g,f)\right]\) and is well defined. As \(G_b\) acts transitively on \(F\) every element in \(B_F\) may be denoted by an equivalence class of the form \([(g,f)]\) for any fixed \(f\in F\). From this it follows that the action of \(G\) is transitive with the isotropy subgroup of \([(e,f)]\) equal to \((G_b)_f\).
\end{proof}
In particular, we remark that when the fibre \(F\) is an affine space, and the action of \(G_b\) on \(F\) is that of an affine representation, then \(B_F\) is an affine bundle. Recall that an affine bundle is one where scalar multiplication and subtraction is defined on the fibres, but unlike a vector bundle, there is no canonical choice of zero section.

\begin{thm}\label{main_thm}
	Let \(G\) be the (special) Euclidean group. The following holds.
	\begin{enumerate}
		\item Adjoint orbits through \((\omega,0)\in\Delta_{\mathfrak{h}}\) are \(G\)-equivariant vector bundles over the corresponding coadjoint orbit through \((L,0)\in\Delta_{\mathfrak{h}^*}\) with fibres isomorphic to \(\Imag\omega\).
		
		\item Coadjoint orbits through \((L,p)\in\overline{\Delta^*}\) are \(G\)-equivariant affine bundles over the corresponding adjoint orbit through \((\omega,v)\in\overline{\Delta}\) with fibres isomorphic to the quotient \(\ker\omega/\ker\tau_p\).
	\end{enumerate}
\end{thm}
\begin{proof}
	Once again, as \(\omega\mapsto L(\omega)\) and \(v\mapsto p(v)\) are \(H\)-equivariant maps, we have that \(H_\omega=H_L\) and \(H_v=H_p\). Therefore, with reference to \eqref{alt_Coad_action}, the orbit \(\orb^G_{L,0}\) may be identified with the adjoint orbit \(\orb^H_\omega\). For Part 1, the desired bundle \(\orb^G_{\omega,0}\rightarrow\orb^G_{L,0}\cong\orb^H_\omega\) is the same as that given in Figure~\ref{orbit_geometry_fig}.

	Before addressing Part 2 we must first determine what the space \(\ker\tau_p\) is equal to. By unpacking the definitions of \(\tau_p\) and \(B_V\), observe that \(x\) belongs to \(\ker\tau_p\) if and only if \(B_V(\xi v,x)=0\) for all \(\xi\in\mathfrak{h}\). For when \(H=SO(n)\) or \(O(n)\), the orbit of \(H\) through \(v\) is a sphere of radius \(B_V(v,v)\). Therefore, \(\ker\tau_p\) is equal to the one-dimensional span of \(v\), and consequently is fixed by \(H_p=H_v\).
	
	For Part 2 write \(b=(\omega,v)\) and let \(B=\orb^G_{\omega,v}\). Consider the principal \(G_b\)-bundle \(G\rightarrow B\) where \(G_b=G_{\omega,v}=(H_\omega\cap H_v)\ltimes\ker\omega\). 
	The group \(H_\omega\cap H_v\) preserves both \(\ker\omega\) and \(\ker\tau_p\), and thus induces an affine representation of \(G_b\) on the quotient \(F=\ker\omega/\ker\tau_p\) given by \((r,d)\cdot[k]=[rk+d]\). This representation is transitive, and the stabiliser of \(f=[0]\) is the subgroup \((H_\omega\cap H_v)\ltimes\ker\tau_p\). By Proposition~\ref{associated_vector}, the associated affine bundle \(B_F\) admits a transitive \(G\)-action with isotropy subgroup \((H_\omega\cap H_v)\ltimes\ker\tau_p\). But this is equal to \(G_{L,p}=(H_p\cap H_L)\ltimes\ker\tau_p\), the stabiliser of the coadjoint orbit through \((L,p)\). Therefore \(B_F\) may be identified with \(\orb^G_{L,p}\) and we are done.
\end{proof}
This theorem may be checked against the example in Figure~\ref{se2} for \(SE(2)\). Here observe that \(\Delta_{\mathfrak{h}}\) is the \(\omega\)-axis, and \(\overline{\Delta}\) is the \(v_xv_y\)-plane minus the origin.

\section{Orbits of the Euclidean group} \label{sec:Euclidean}
It is well known that the adjoint orbits of the unitary group \(U(n)\) are equal to the manifolds of complex flags in \(\mathbb{C}^n\). However, the literature does not appear to address the exact nature of the adjoint orbits of the orthogonal group. In this section we review the definition of a flag manifold and introduce new \emph{Hermitian} and \emph{affine} flags endowed with additional structure. The adjoint and coadjoint orbits of the orthogonal and Euclidean groups turn out to be examples of such manifolds. We begin by recalling a geometric definition for the Euclidean group which will serve us throughout this section.
\begin{defn}
	For an affine space \(A\), the \emph{Euclidean group} \(E(A)\) is the group of affine-linear isomorphisms which preserves a given Euclidean distance. If the group also preserves an orientation on \(A\) we write the \emph{special Euclidean group} as \(SE(A)\).
\end{defn}
\subsection{Linear flags}
\begin{defn}
	A \emph{flag} \(F\) in \(V=\mathbb{R}^n\) is a strictly ascending sequence of subspaces beginning with \(\{0\}\) and ending with \(V\) given by
	\begin{equation}
	\label{flag_ordinary}
	F=\left(\{0\}=E_0\subsetneq E_1\subsetneq\dots\subsetneq E_k=V\right).
	\end{equation}
	We will call the subspaces \(E_i\) the \emph{flag subspaces}. The \emph{rank} of the flag \(F\) is the number \(k\) of non-zero flag subspaces. Let \(d_i\) denote the dimension of \(E_i/E_{i-1}\); the tuple \(\sigma=(d_1,\dots,d_k)\) is the \emph{signature} of the flag \(F\). The set of all flags in \(V\) of a given signature \(\sigma\) will be written as \(\F(\sigma)\). These sets generalise the notion of projective spaces and Grassmannians. For example, in our notation we have \(\mathbb{R}P^{n-1}=\F(1,n-1)\) and \(\Gr_\mathbb{R}(k;n-k)=\F(k,n-k)\).
\end{defn}

We now suppose that \(V\) is equipped with an inner product. A flag \(F\) now determines a unique ordered sequence of mutually orthogonal subspaces \((V_1,\dots, V_k)\) where \(V_1=E_1\) and \(E_{j+1}=E_j\oplus V_{j+1}\) for \(1\le j\le k\). We will call the subspaces \(V_j\) the \emph{flag components} of \(F\), and note that \(\dim V_j=d_j\).  Conversely, observe that an ordered sequence of orthogonal subspaces which span \(V\) uniquely determines a flag with these subspaces as flag components.

These flags may be equipped with additional structure. For instance, we may prescribe an orientation on a given flag component \(V_j\), or equivalently on the quotient \(E_j/E_{j-1}\). Should a flag possess this attribute we will write the signature as \((d_1,\dots,\widetilde{d_j},\dots,d_k)\) where the tilde indicates that the flag component \(V_j\) receives an orientation. Flags may also be endowed with an orthogonal complex structure on a given \(V_j\), or equivalently on \(E_j/E_{j-1}\); that is, a linear map \(\mathbb{J}_j\colon V_j\rightarrow V_j\) with \(\mathbb{J}_j^2=-I\) which preserves the inner product. We will write the signature of such a flag as  \((d_1,\dots,d_j^\mathbb{C},\dots,d_k)\), where the raised \(\mathbb{C}\) indicates that \(V_j\) (which is necessarily even-dimensional) has been endowed with an orthogonal complex structure. We will refer to these flags with additional structure as \emph{Hermitian flags}.

The set of all such flags of a given signature admits a transitive action by the orthogonal group \(O(V)\) preserving the inner product. For any \(r\) we define \(rF\) to be the flag whose flag subspaces are given by \(rE_j\). Equivalently, the flag components \(V_j\) determined by \(F\) are sent to \(rV_j\). In addition, if a flag component \(V_j\) is equipped with either an orientation or complex structure \(\mathbb{J}_j\), then the corresponding orientation on \(rV_j\) is the induced orientation, and the complex structure given by \(r\circ\mathbb{J}_j\circ r^{-1}\). This action is transitive and allow us to write these sets as homogeneous spaces
\begin{align}
\F(d_1,\dots,\widetilde{d_j},\dots,d_k^\mathbb{C})&=\frac{O(V)}{O(V_1)\times\dots \times SO(V_j)\times\dots \times U(V_k)}.
\end{align}
Let \(F\) be a flag of rank \(k\) as in equation~\eqref{flag_ordinary} of composite signature \((\sigma,\tau)\); here the signature \(\sigma\) refers to the first \(l\) flag subspaces of \(F\) and \(\tau\) the remaining \(k-l\) subspaces for a given \(l\le k\). There is a bundle map \(\F(\sigma,\tau)\longrightarrow\F(\sigma)\) given by sending the flag \(F\) to the flag \(E_0\subsetneq E_1\subsetneq\dots\subsetneq E_l\subsetneq V\); effectively `forgetting' the higher flag subspaces given by \(\tau\). The fibre containing \(F\) is equal to the flag manifold of flags in \(V/E_l\) of signature \(\tau\). We thus have the following equivariant fibre-bundle for flag manifolds:
\begin{equation}\label{ord_flag_bundle}
	\begin{tikzcd}
	\F(\tau)\arrow[r,hook] & \F(\sigma,\tau)\arrow[d,two heads] \\ & \F(\sigma)
\end{tikzcd}
\end{equation}
\begin{thm}[Orbits of $O(n)$] \label{on_orbits}
Any adjoint or coadjoint orbit of \(O(n)\) is equivariantly diffeomorphic to a manifold of Hermitian flags in \(\mathbb{R}^n\).
\end{thm}
\begin{proof}
	Let \(\omega\) belong to a given adjoint orbit through \(\mathfrak{h}=\mathfrak{so}(n)\). Since \(\omega\) is a skew-symmetric matrix it admits an orthogonal eigenspace decomposition with respect to its complexified action on \(\mathbb{C}^n\), along with purely imaginary eigenvalues. Moreover, non-zero eigenvalues appear in opposite pairs; that is, if \(i\lambda_j\) is an eigenvalue of \(\omega\) then so is \(-i\lambda_j\) with the same multiplicity. Let \(E_{\pm \lambda_j}\subset\mathbb{C}^n\) denote the eigenspaces for a distinct non-zero eigenvalue pair \(\pm i\lambda_j\). Observe that the action of \(\omega\) restricted to the real subspace \(V_j=\left(E_{+\lambda_j}\oplus E_{-\lambda_j}\right)\cap\mathbb{R}^n\) squares to minus the identity multiplied by \(\lambda_j^2\). Order the eigenvalues of \(\omega\) so that the non-zero distinct eigenvalue pairs \(\pm i\lambda_1,\dots,\pm i\lambda_k\) satisfy  \(|\lambda_1|<\dots<|\lambda_k|\) and where, should \(\ker\omega\ne\{0\}\), write the zero eigenvalue as \(\lambda_0=0\). It can now be seen that \(\omega\) uniquely determines a flag \(F\) in \(\mathbb{R}^n\) with flag components given by 
	\begin{equation}
	\label{son_flag}
	\left(\ker\omega,V_1,\dots,V_k\right),
	\end{equation}
	together with orthogonal complex structures on each \(V_j\) given by restriction of \(\lambda_j^{-1}\omega\). This correspondence between elements in \(\mathfrak{so}(n)\) and Hermitian flags is equivariant with respect to \(O(n)\), and therefore defines an equivariant diffeomorphism between the adjoint orbit through \(\omega\) with the manifold \(\F(d_0,d_1^\mathbb{C},\dots,d_k^\mathbb{C})\). Finally, since the isomorphism \(\omega\mapsto L\) between \(\mathfrak{so}(n)\) and \(\mathfrak{so}(n)^*\) is \(O(n)\)-equivariant, the coadjoint orbit through \(L=L(\omega)\) may be identified with the adjoint orbit through \(\omega\).
\end{proof}
More generally, the adjoint orbits of a semisimple Lie group go by the names of \emph{generalised flag varieties} or simply just \emph{flag manifolds} \cite{alek}. Though we do not show it here, the adjoint orbits of the compact symplectic group \(Sp(n)\), the group of isomorphisms of quaternionic `vector space' \(\mathbb{H}^n\) preserving a Hermitian form,  may also be described as a manifold of flags. Analogously to the case for the orthogonal group, the flag manifolds in question consist of quaternionic flags with a designated choice of complex structure assigned to certain flag components.

\subsection{Affine flags}
\begin{defn}
	An \emph{affine flag} in \(V\) is a strictly ascending sequence of affine subspaces terminating with \(V\).
	\begin{equation}
	\label{ord_aff_flag}
	F=\left(A_1\subsetneq\dots\subsetneq A_k=V\right)
	\end{equation}
	The affine subspaces \(A_j\) will also be referred to as the flag subspaces of \(F\). Recall that any affine subspace \(A_j\) determines an \emph{associated subspace} \([A_j]\) given by \(A_j-v\) for any \(v\in A_j\), and that this does not depend on the choice of \(v\). In this way, an affine flag \(F\) determines an \emph{associated flag} \([F]\) given as in equation~\eqref{flag_ordinary} where \(E_j=[A_j]\). If \([F]\) has signature \((d_0,d_1,\dots,d_k)\), then \(F\) is defined to have signature \((d_0;d_1,\dots,d_k)\). The set of all affine flags of a given signature \(\sigma\) will be written as \(\aff(\sigma)\).
\end{defn}
The manifold \(\aff(\sigma)\) admits a transitive action by the Euclidean group. Here \((r,d)\in E(V)\) sends the flag \(F\) to \(rF+d\), the affine flag obtained by sending each flag subspace \(A_j\) to \(rA_j+d\). The induced action of \((r,d)\) on the associated flag \([F]\) is \(r[F]\), that is \([rF+d]=r[F]\). It follows that these manifolds are homogeneous spaces
\begin{equation}
\label{aff_hermitian_flag}
\aff(\widetilde{d_1};\dots d_j,\dots,d_k^\mathbb{C})=\frac{E(V)}{SE(A_1)\times\dots\times O(V_j)\times\dots\times U(V_k)}.
\end{equation}
 The map \(F\mapsto [F]\) gives an equivariant vector bundle over \(\F(\sigma)\) whose fibre above \([F]\) is equal to all distinct translates of the subspace \(E_1\) (which we might like to call the \emph{flag pole} of the flag \(F\)). This is isomorphic to the vector space \(V/E_1\).
\begin{equation}\label{aff_flag_vbundle}
\begin{tikzcd}
V/E_1\arrow[r,hook] & \aff(\sigma)\arrow[d,two heads] \\ & \F(\sigma)
\end{tikzcd}
\end{equation}
The zero section of this vector bundle is the subset of ordinary linear flags viewed as affine flags; that is, those affine flags \(F\) with \([F]=F\).

In much the same way as we constructed the fibre bundle in \eqref{ord_flag_bundle}, we may apply the same reasoning to affine flags and establish the equivariant fibre bundle
\begin{equation}\label{aff_flag_bundle}
\begin{tikzcd}
\F(\tau)\arrow[r,hook] & \aff(\sigma,\tau)\arrow[d,two heads] \\ & \aff(\sigma)
\end{tikzcd}
\end{equation}
\begin{thm}[Coadjoint orbits of $E(n)$]  \label{en_coad_orbits}
	All coadjoint orbits of \(G=E(n)\) intersect the sets \(\Delta_{\mathfrak{h}^*}\) and \(\overline{\Delta^*}\) as given in \eqref{Delta_defn}. Let \(L=L(\omega)\), where \(\omega\) is the element in \(\mathfrak{so}(n)\) determining the flag components in \eqref{son_flag}. The orbit through \((L,0)\in\Delta_{\mathfrak{h}^*}\) is equivariantly diffeomorphic to \(\F(d_0,d_1^\mathbb{C},\dots,d_k^\mathbb{C})\), the orbit through \((L,p)\in\overline{\Delta^*}\) is equivariantly diffeomorphic to \(\aff(\widetilde{1};d_0-1,d_1^\mathbb{C},\dots,d_k^\mathbb{C})\) and the orbit \(\orb^H_{L,p}\) is \(H\)-equivariantly diffeomorphic to \(\F(\widetilde{1},d_0-1,d_1^\mathbb{C},\dots,d_k^\mathbb{C})\).
\end{thm}
\begin{proof}
	From the coadjoint action given in \eqref{alt_Coad_action}, observe that the orbit through \((L,0)\) is equal to the orbit \(\orb^H_L\). This orbit is identified with \(\orb^H_\omega\) and therefore the result for \(\Delta_{\mathfrak{h}^*}\) follows from Theorem~\ref{on_orbits}.
	
	For \((L,p)\in\overline{\Delta^*}\) the isotropy subgroup is given in \eqref{Coad_proper_isotropygroup} by \(G_{L,p}=(H_p\cap H_L)\ltimes\ker\tau_p\). From the proof of Theorem~\ref{main_thm}, this is equal to \((H_v\cap H_\omega)\ltimes\spoon \{v\}\), where \(v\) is the non-zero element in \(\ker\omega\) with \(p=p(v)\). Therefore, the subgroup \(H_\omega\cap H_v\) consists of those elements which preserve the flag components
	\begin{equation}
	\label{finerflag}
	(\spoon\{v\},K_0,V_1,\dots, V_k)
	\end{equation}
	together with the complex structures defined on each \(V_j\). Here the one-dimension flag component \(\spoon \{v\}\) is oriented (in the direction of \(+v\)) and \(K_0\) is the (possibly trivial) \((d_0-1)\)-dimensional subspace orthogonal to \(v\) satisfying \(\ker\omega=\spoon v\oplus K_0\). From equation~\eqref{aff_hermitian_flag} we recognise that the group \(G_{L,p}\) is the stabiliser subgroup of the affine flag \(F\) as in \eqref{ord_aff_flag} with associated flag components given by \eqref{finerflag}, and for where \(A_1=\spoon \{v\}\). There is therefore an \(E(n)\)-equivariant diffeomorphism from the orbit through \((L,p)\in\overline{\Delta^*}\) with \(\aff(\widetilde{1};d_0-1,d_1^\mathbb{C},\dots,d_k^\mathbb{C})\) given by sending \(\Ad_{(r,d)}^*(L,p)\) to \(rF+d\). The result for the orbit \(\orb^H_{L,p}\) is obtained similarly, by remarking that the isotropy subgroup is \(H_\omega\cap H_v\).
\end{proof}
Armed with this result, we can test it out to obtain the well-known coadjoint orbits of \(E(3)\): the point orbit, spheres, and tangent bundles to spheres \cite[Theorem 4.4.1]{bigstages}. Before showing this, we recall that the manifold of oriented affine lines \(\aff(\widetilde{1};n-1)\) may canonically be identified with the tangent bundle to the unit sphere \(S^{n-1}\subset\mathbb{R}^n\). An oriented affine line \(l\) determines an associated oriented one-dimensional subspace \([l]\). This in turn determines a unique point \(v\in S^{n-1}\). The correspondence identifies the line \(l\) with the intersection of \(l\) with the tangent space \(T_vS^{n-1}\subset\mathbb{R}^n\).

\begin{cor}\label{cor1}
	For \(G=E(3)\), the coadjoint orbit through \((L,0)\) in \(\Delta_{\mathfrak{h}^*}\) is: the point orbit if \(L=0\), and the sphere \(\F(1,2^\mathbb{C})\cong S^2\) for \(L\ne 0\). For \((L,p)\) in \(\overline{\Delta^*}\), the orbit is: a single sphere tangent bundle \(\aff(\widetilde{1};2)\cong TS^2\) if \(L=0\), or two disjoint sphere tangent bundles \(\aff(\widetilde{1};2^\mathbb{C})\cong TS^2\sqcup TS^2\) for when \(L\ne 0\).
\end{cor}

\begin{proof}
	For \(\omega\in\mathfrak{so}(3)\), either \(\omega=0\) or \(\omega\) is non-zero and admits flag components \((\ker\omega, V_1)\). The kernel \(\ker\omega\) is necessarily one-dimensional, and \(V_1\) is equipped with a complex structure. Now directly apply Theorem~\ref{en_coad_orbits}.
\end{proof}
\subsection{Affine flags with grain}
\begin{defn}\label{grain_def}
	An \emph{affine flag with grain} is defined to be an affine flag \(F\) as in \eqref{ord_aff_flag} but with an additional oriented flag space \(V_0\), the \emph{grain}, prepended to the associated flag
	\[
	[F]=\left(\{0\}\subsetneq V_0\subsetneq E_1\subsetneq\dots\subsetneq E_k=V\right)
	\]
	where each \(E_j\) is the associated subspace \([A_j]\). If \([F]\) has signature \((\widetilde{d_0},d_1,\dots,d_k)\) then we write the signature of \(F\) as \(([\widetilde{d_0},d_1];\dots,d_k)\).
\end{defn}
 The affine subspace \(A_1\) of such a flag may be thought of as a union of parallel subspaces of the form \(V_0+v\) as \(v\) ranges over \(A_1\). Our choice of terminology is to invoke the image of parallel lines of grain running along the length of a wooden plank. The manifold of affine flags with grain also admits a transitive action under the Euclidean group. Here the stabiliser of a flag must preserve all flag subspaces \(A_j\) as well as the oriented grain \(V_0\). The manifold of all such flags is a homogeneous space given by
\begin{equation}\label{aff_flag_with_grain_isotropy}
\aff([\widetilde{d_0},d_1];\dots, d_k)=\frac{E(V)}{E(A_1)_{[V_0]}\times\dots\times O(V_k)}.
\end{equation}
The notation \(E(A_1)_{[V_0]}\) denotes the subgroup of \(E(A_1)\) for which \([rV_0+d]=V_0\) and preserves the orientation of \(V_0\). For example, if \(V_0\) is the oriented span of the line \(v\), then \(E(n)_{[V_0]}=O(n)_v\ltimes\mathbb{R}^n\).

There are three natural bundle structures on the manifold of affine flags with grain which concern us. The first bundle map sends the affine flag \(F\) with grain \(V_0\) to the ordinary affine flag obtained by forgetting the grain subspace \(V_0\). The fibre is then equal to the space of all oriented \(d_0\)-dimensional subspaces \(V_0\) inside \(E_1\). This gives us the equivariant fibre bundle 
\begin{equation}\label{adjoint_left}
\begin{tikzcd}
\F(\widetilde{d_0},d_1)\arrow[r,hook] & \aff([\widetilde{d_0},d_1];\sigma)\arrow[d,two heads] \\ & \aff(d_0+d_1;\sigma).
\end{tikzcd}
\end{equation}
The second bundle structure is the vector bundle map sending \(F\) to the associated flag \([F]\). The fibre in this case is once again equal to all distinct translates of \(A_1\), itself isomorphic to the vector space \(V/E_1\). We thus have the equivariant vector bundle
\begin{equation}\label{blah2}
\begin{tikzcd}
V/E_1\arrow[r,hook] & \aff([\widetilde{d_0},d_1];\sigma)\arrow[d,two heads] \\ & \F(\widetilde{d_0},d_1,\sigma).
\end{tikzcd}
\end{equation}
The third and final bundle map sends an ordinary affine flag \(
F=\left(A_0\subsetneq A_1\subsetneq\dots\subsetneq A_k=V\right)
\),
to the affine flag with grain \(F=A_1\subsetneq\dots\subsetneq A_k=V\) with associated flag \([F]=\left(E_0\subsetneq E_1\subsetneq\dots\subsetneq E_k=V\right)\); thus the first space \(A_0\) is forgotten yet its associated space \(E_0=[A_0]\) becomes the grain. The fibre is equal to all distinct translates of  \(A_0\) within the space \(A_1\). This is isomorphic to the affine space \(A_1/A_2\) and so we obtain the equivariant affine bundle
\begin{equation}\label{coad_bijbundle}
\begin{tikzcd}
A_1/A_2\arrow[r,hook] & \aff(\widetilde{d_0};d_1,\sigma)\arrow[d,two heads] \\ & \aff([\widetilde{d_0},d_1];\sigma).
\end{tikzcd}
\end{equation}

\begin{thm}[Adjoint orbits of \(E(n)\)]  \label{en_adjoint_orbits}
	All adjoint orbits of \(G=E(n)\) intersect the sets \(\Delta_{\mathfrak{h}}\) and \(\overline{\Delta}\) as given in \eqref{Delta_defn}. Let \(\omega\) be an element in \(\mathfrak{so}(n)\) determining the flag components in \eqref{son_flag}. The orbit through \((\omega,0)\in\Delta_{\mathfrak{h}}\) is equivariantly diffeomorphic to \(\aff(d_0;d_1^\mathbb{C},\dots,d_k^\mathbb{C})\), the orbit through \((\omega,v)\in\overline{\Delta}\) is equivariantly diffeomorphic to \(\aff([\widetilde{1},d_0-1]; d_1^\mathbb{C},\dots,d_k^\mathbb{C})\) and the orbit \(\orb^H_{\omega,v}\) is \(H\)-equivariantly diffeomorphic to \(\F(\widetilde{1},d_0-1, d_1^\mathbb{C},\dots,d_k^\mathbb{C})\).
\end{thm}
\begin{proof}
	From \eqref{Ad_isotropygroup_proper} we have \(G_{\omega,0}=H_\omega\ltimes\ker\omega\). From \eqref{aff_hermitian_flag} we recognise that this is equal to the isotropy group of an affine flag \(F\) as in \eqref{ord_aff_flag} with flag components given by \eqref{son_flag} and \(A_1=\ker\omega\). It follows that there is an equivariant diffeomorphism between this orbit and \(\aff(d_0;d_1^\mathbb{C},\dots,d_k^\mathbb{C})\) given by sending \(\Ad_{(r,d)}(\omega,0)\) to the flag \(rF+d\).
	
	For \((\omega,v)\in\overline{\Delta}\), since the point is proper we have that \(G_{\omega,v}=(H_v\cap H_\omega)\ltimes\ker\omega\). Therefore, from \eqref{aff_flag_with_grain_isotropy} we recognise that \(G_{\omega,v}\) is equal to the isotropy subgroup fixing the affine flag with grain \(F\) as given in Definition~\ref{grain_def}, whose associated flag is determined by the flag components in \eqref{finerflag}, and with \(A_1=\ker\omega\). The diffeomorphism is then given by sending \(\Ad_{(r,d)}(\omega,v)\) to the flag \(rF+d\). The result for the orbit \(\orb^H_{\omega,v}\) is obtained similarly by remarking that the isotropy subgroup is \(H_\omega\cap H_v\).
\end{proof}
\begin{cor}\label{cor2}
	For \(G=E(3)\), the adjoint orbit through \((\omega,0)\) in \(\Delta_{\mathfrak{h}}\) is: the point orbit if \(\omega=0\), and a sphere tangent bundle \(\aff(1;2^\mathbb{C})\cong TS^2\) for \(\omega\ne 0\). For \((\omega,v)\) in \(\overline{\Delta}\), the orbit is: a sphere \(\aff([\widetilde{1};2])\cong S^2\) if \(\omega=0\), or two disjoint sphere tangent bundles \(\aff(\widetilde{1};2^\mathbb{C})\cong TS^2\sqcup TS^2\) for when \(\omega\ne 0\).
\end{cor}
If we identify \(\mathfrak{so}(3)\) with \(\mathbb{R}^3\) in the standard way, then the form \(\langle(\omega_1,v_1),(\omega_2,v_2)\rangle=\omega_1^Tv_2+v_1^T\omega_2\) defines an invariant non-degenerate symmetric form on \(\mathfrak{se}(3)\). Therefore, the adjoint and coadjoint representations of \(E(3)\) are isomorphic, and thus the orbits are identical. This confers with the results from Corollaries~\ref{cor1} and~\ref{cor2}. However, notice that the orbit bijection from Theorem~\ref{main_thm} does not pair identical orbits together, but that they are indeed homotopic.
\subsection{The geometry of the Euclidean group orbits}

As the adjoint and coadjoint orbits of \(E(n)\) are flag manifolds, the orbit geometry described earlier concerning bundles between orbits and submanifolds may be reinterpreted as maps between flag manifolds. In Figure~\ref{en_diag} we recast the diagrams from Figure~\ref{orbit_geometry_fig} by replacing the generic orbits through \(\overline{\Delta}\) and \(\overline{\Delta^*}\) with the corresponding Hermitian flag manifolds using Theorems~\ref{en_adjoint_orbits} and \ref{en_coad_orbits}. The bundle maps between the orbits precisely correspond to the natural flag bundles introduced before.

The orbit bijection result from Theorem~\ref{main_thm} for the Euclidean group may also be recast in terms of flag manifolds. For elements in \(\Delta_{\mathfrak{h}}\) and \(\Delta_{\mathfrak{h}^*}\), the vector bundle \(\orb^G_{\omega,0}\longrightarrow\orb^G_{L,0}\) between two bijected orbits is the vector bundle 

\begin{equation}
\aff(d_0;d_1^\mathbb{C},\dots,d_k^\mathbb{C})\longrightarrow\F(d_0,d_1^\mathbb{C},\dots,d_k^\mathbb{C})
\end{equation}
given in \eqref{aff_flag_vbundle} with fibres isomorphic to \(V/\ker\omega\cong\Imag\omega\cong\mathbb{R}^{n-d_0}\). Recall that these fibres may be identified with the translates of the space \(\ker\omega\). For elements in \(\overline{\Delta}\) and \(\overline{\Delta^*}\), the affine bundle \(\orb^G_{L,p}\longrightarrow\orb^G_{\omega,v}\) between two bijected orbits is the affine bundle
\begin{equation}
\aff(\widetilde{1};d_0-1,d_1^\mathbb{C},\dots,d_k^\mathbb{C})\longrightarrow\aff([\widetilde{1},d_0-1];d_1^\mathbb{C},\dots,d_k^\mathbb{C})
\end{equation}
given in \eqref{coad_bijbundle} with fibres isomorphic to the affine quotient \(\ker\omega/\ker\tau_p\cong\mathbb{R}^{d_0-1}\). Recall that these fibres may be identified with the translates of \(\ker\tau_p\) within \(\ker\omega\). As these bundles all have contractible fibres, we once again arrive at the result that bijected orbits have the same homotopy type. 
\def\D{\mathcal{D}}
\def\R{\mathbb{R}}

From Theorem~\ref{orbit_geometry_thm} all generic coadjoint orbits through \((L,p)\in\overline{\Delta^*}\) are equivariant fibre bundles over the orbit through \((0,p)\), which by Theorem~\ref{en_coad_orbits} may be identified with \(\aff(\widetilde{1};n-1)\). This bundle map corresponds to the flag bundle in \eqref{aff_flag_bundle} which sends the affine flag to the one-dimensional oriented affine line, or flag pole. This orbit is the well-known symplectic manifold $\D_n$ of oriented lines in $\R^n$. The invariant symplectic form is unique up to scalar multiple and can be described as follows. First (similar to the usual argument for Grassmannians), every line near $\ell\in\D_n$ is the graph of an affine map $\ell\to\ell^\perp$, whence we can identify 
	$$ T_\ell\D_n \simeq \mathrm{Aff}(\ell,\ell^\perp).$$
	An element of this space can be given by $\mathbf{a}+s\mathbf{b}$ where $s$ is the parameter along $\ell$ and $\mathbf{a},\mathbf{b}$ are elements of $\ell^\perp$ (or more canonically of $\R^n/\ell$). Given two such affine linear maps, one forms the skew product
	$$\omega(\mathbf{a}+s\mathbf{b},\mathbf{a}'+s\mathbf{b}') = \mathbf{a}\cdot\mathbf{b'} - \mathbf{a'}\cdot\mathbf{b}.$$
	This is a well-defined expression (independent of choice of point $s=0$ on $\ell$) and defines a natural symplectic structure on $\D_n$.  On the other hand, it does depend on the speed of parametrization of $\ell$, and this gives the scalar multiple relating different invariant symplectic structures. 
	
	The orbit \(\D_n\) may be thought of as being a fundamental orbit, since all generic orbits fibre over it. Therefore, using this result from Theorem~\ref{orbit_geometry_thm} in combination with Theorems~\ref{on_orbits} and \ref{en_coad_orbits}, we can establish the following result concerning the symplectic geometry of flag manifolds.

	\begin{cor}
		Let \(d_0,d_1,\dots,d_k\) be positive integers with each \(d_j\) even for \(j>0\). The flag manifolds \(	\aff(\widetilde{1};d_0-1,d_1^\mathbb{C},\dots,d_k^\mathbb{C})\) and \(\F(d_0-1,d_1^\mathbb{C},\dots,d_k^\mathbb{C})\) may be given a symplectic form which is invariant under the actions of \(E(n)\) and \(O(n)\) respectively. With respect to this symplectic structure the \(E(n)\)-equivariant bundle
		\[
		\aff(\widetilde{1};d_0-1,d_1^\mathbb{C},\dots,d_k^\mathbb{C})\longrightarrow\aff(\widetilde{1};n-1) = \mathcal{D}_n
		\]
		given in \eqref{aff_flag_bundle} is a symplectic fibration with fibres symplectomorphic to \(\F(d_0-1,d_1^\mathbb{C},\dots,d_k^\mathbb{C})\), and the \(E(n)\)-equivariant vector bundle 
		\[
		\aff(\widetilde{1};d_0-1,d_1^\mathbb{C},\dots,d_k^\mathbb{C})\longrightarrow\F(\widetilde{1},d_0-1,d_1^\mathbb{C},\dots,d_k^\mathbb{C})
		\] 
		given in \eqref{aff_flag_vbundle} has isotropic fibres.
	\end{cor}

	We conclude by remarking that, since any given group \(G\) of Euclidean type is a subgroup of \(E(n)\) for large enough \(n\), every adjoint and coadjoint orbit of \(G\) will be a submanifold of the flag manifolds that we have exhibited for \(E(n)\). As a particular example, for the special Euclidean group \(SE(n)\), the orbits are the connected components of the orbits of \(E(n)\). These also admit a geometric description in terms of flag manifolds.

\begin{figure}[t]

\begin{subfigure}[h]{\textwidth}
	\caption{Adjoint orbit diagram}
	\centering
		\begin{tikzpicture}[scale=0.6]
	\path[use as bounding box] (-7,-2) rectangle (15,7);
	%%ALL the diagram arrows and spaces
	\draw[-{Latex}] (0.9,0) node [left,scale=1]{$\aff(d_0;d_1^\mathbb{C},\dots,d_k^\mathbb{C})$} -- (6.5,0) node[right,scale=1]{$\F(d_0,d_1^\mathbb{C},\dots,d_k^\mathbb{C})$};
	\draw[-{Latex} ] (9,5.4) node[scale=1,above] {$\F(\widetilde{1},d_0-1,d_1^\mathbb{C},\dots,d_k^\mathbb{C})$} -- (9,.6);
	\draw[ -{Latex}  ] (-1.5,5.4)  -- (-1.5,0.8);	
	\draw[-{Latex}] (1.4,6)node[left, scale=1] {$\aff([\widetilde{1},d_0-1];d_1^\mathbb{C},\dots,d_k^\mathbb{C})$} -- (5.8,6);
	%%% the annotation labels
	\node at (4,-0.5){$\mathbb{R}^{n-d_0}$};	
	\node at (4,6.5){$\mathbb{R}^{n-d_0}$};	
	\node at (11,3.5){ $\F(\widetilde{1},d_0-1)$};		
	\node at (-3.5,3.5){ $\F(\widetilde{1},d_0-1)$};
	\end{tikzpicture}	
\end{subfigure}
\begin{subfigure}[h]{\textwidth}
	\caption{Coadjoint orbit diagram}
		\centering
\begin{tikzpicture}[scale=0.6]
\path[use as bounding box] (-7,-2) rectangle (15,7);
%%ALL the diagram arrows and spaces
\draw[-{Latex}] (0.9,0) node [left,scale=1]{$\aff(\widetilde{1};n-1)$} -- (7,0) node[right,scale=1]{$\F(\widetilde{1},n-1)$};
\draw[-{Latex} ] (9,5.4) node[scale=1,above] {$\F(\widetilde{1},d_0-1,d_1^\mathbb{C},\dots,d_k^\mathbb{C})$} -- (9,.6);
\draw[ -{Latex}  ] (-1.5,5.4)  -- (-1.5,0.8);	
\draw[-{Latex}] (1.4,6)node[left, scale=1] {$\aff(\widetilde{1};d_0-1,d_1^\mathbb{C},\dots,d_k^\mathbb{C})$} -- (5.8,6);
%%%%%%%%%
%\node[scale=1.4] at (-1.5,-1){\rotatebox[]{90}{$\cong$}};
%\node[scale=1] at (-1.3,-1.8){$T^*S^{n-1}$};
%\node[scale=1.4] at (9,-1){\rotatebox[]{90}{$\cong$}};
%\node[scale=1] at (9,-1.8){$S^{n-1}$};	
%%% the annotation labels
\node at (4,-0.5){$\mathbb{R}^{n-1}$};	
\node at (4,6.5){$\mathbb{R}^{n-1}$};	
\node at (12,3.5){ $\F(d_0-1,d_1^\mathbb{C},\dots,d_k^\mathbb{C})$};		
\node at (-4.7,3.5){ $\F(d_0-1,d_1^\mathbb{C},\dots,d_k^\mathbb{C})$};
\end{tikzpicture}
\end{subfigure}
\caption{\label{en_diag} }
\end{figure}

\small  %\setlength{\parskip}{-6pt}

\bigskip
\setlength{\parindent}{0pt}

%\noindent\begin{minipage}[t]{0.4\textwidth}
%PA: philip.arathoon@manchester.ac.uk
%
%JM: j.montaldi@mancehster.ac.uk
%\end{minipage}
%\hskip1cm
%\begin{minipage}[t]{0.5\textwidth}
%School of Mathematics \\
%University of Manchester \\
%Manchester, M13 9PL, UK.
%\end{minipage}

PA: philip.arathoon@manchester.ac.uk

JM: j.montaldi@mancehster.ac.uk

\medskip
\hfill \it
School of Mathematics, 
University of Manchester,
Manchester, M13 9PL, UK.


\begin{thebibliography}{9}

	\bibitem{Arathoon} Arathoon, P., \emph{In preparation}.
	
	\bibitem{bott}
	Atiyah, M.F. and Luke, G.L., 1979. \emph{Representation theory of Lie groups} (Vol. 34). Cambridge University Press.
	
	\bibitem{alek}
	Alekseevsky, D.V., 1997. Flag manifolds. \emph{Sbornik Radova}, 11, pp.3--35.
	
	\bibitem{baguis}
	Baguis, P., 1998. Semidirect products and the Pukanszky condition. \emph{Journal of Geometry and Physics}, 25(3-4), pp.245--270.
	
	\bibitem{cush06}
	Cushman, R. and Van Der Kallen, W., 2006. Adjoint and coadjoint orbits of the Poincar\'e group. \emph{Acta Applicandae Mathematica}, 90(1-2), pp.65--89.
	
%	\bibitem{donaldson}
%	Donaldson, S.K., 2011. Lectures on Lie groups and geometry. \emph{Representation Theory}, 29, pp.4--1.
	
	\bibitem{stern}
	Guillemin, V. and Sternberg, S., 1990. \emph{Symplectic techniques in physics}. Cambridge university press.
	
	\bibitem{bigstages}
	Marsden, J.E., Misio\l ek, G., Ortega, J.P., Perlmutter, M. and Ratiu, T.S., 2007. \emph{Hamiltonian reduction by stages} (Vol. 1913). Berlin: Springer.
	
	\bibitem{rawnsley}
	Rawnsley, J.H., 1975, September. Representations of a semi-direct product by quantization. \emph{Mathematical Proceedings of the Cambridge Philosophical Society} (Vol. 78, No. 2, pp. 345--350). Cambridge University Press.
\end{thebibliography}
\end{document}